\documentclass[11pt,reqno,a4paper]{amsart}
\usepackage{euscript}

\newtheorem{theorem}{Theorem}
\newtheorem{proposition}[theorem]{Proposition}
\newtheorem{lemma}{Lemma}

\renewcommand{\epsilon}{\varepsilon}
\DeclareMathOperator{\Ima}{Im}
\DeclareMathOperator{\Ker}{Ker}
\def\Id{\text{\rm Id}}
\def\cA{\EuScript{A}}
\def\N{\mathbb{N}}

\def\R{\mathbb{R}}

\begin{document}

\title[A version of the theorem of Johnson]{A version of the theorem of Johnson, Palmer and Sell for quasicompact cocycles}

\begin{abstract}
The well-known theorem of Johnson, Palmer and Sell asserts that the endpoints of the Sacker--Sell spectrum of a given cocycle of invertible matrices over a topological dynamical system 
$(M, f)$ are realized as Lyapunov exponents with respect to some ergodic invariant probability measure for $f$. In this note we establish the version of this result for quasicompact cocycles
 of operators acting on an arbitrary Banach space.
\end{abstract}

\begin{thanks}
{ D.D. was supported in part by an Australian Research Council Discovery Project DP150100017 and by  Croatian Science Foundation under the project IP-2014-09-2285}
\end{thanks}

\author{Davor Dragi\v cevi\'c}
\address{Department of Mathematics, University of Rijeka, Croatia}
\email{ddragicevic@math.uniri.hr}


\maketitle
\section{Introduction}
In their landmark paper~\cite{SS}, Sacker and Sell introduced the notion of (what is now called) the
Sacker--Sell spectrum for cocycles over topological dynamical systems and they described
all possible structures of the spectrum. Furthermore, they have obtained several results which indicated that there exists a strong  connection between their spectral theory
and the theory of Lyapunov exponents which plays a central role in the modern dynamical systems theory. Those developments culminated with the remarkable paper
by Johnson, Palmer and Sell~\cite{JPS}, where the authors proved that the endpoints of the Sacker-Sell spectrum of a given cocycle are realized as Lyapunov exponents of the cocycle
with respect to some ergodic probability measure which is invariant for the base space on which the
cocycle acts.

We emphasize that the results in~\cite{JPS} and~\cite{SS} deal with invertible cocycles on a finite-dimensional space. More recently, Magalh\~aes~\cite{LM} developed the Sacker-Sell spectral
theory for compact cocycles on an arbitrary Banach space and Schreiber~\cite{S} established the version of the theorem of Johnson, Palmer and Sell in this setting under the additional 
assumption that the cocycle is injective. 

In the present paper, we go step further and we develop the Sacker-Sell spectral theory and establish the version of the theorem of  Johnson, Palmer and Sell  for
not  not necessarily injective quasicompact cocycles acting 
on an arbitrary Banach space. The principal motivation for this work are the most recent versions~\cite{PT, FLQ, GTQ, AB} (that build on the earlier work of Ruelle~\cite{R} and Ma\~n\'e~\cite{mane}) of the classical Oseledets multiplicative ergodic theorem~\cite{Osel}
that work under the assumption that the cocycle is quasicompact. These developments culminated in a recent remarkable paper by Blumenthal and Young~\cite{BY} in which the authors have extended
a large part of the modern smooth ergodic theory to the case of maps $f$ acting on arbitrary Banach spaces. The results in~\cite{BY} are valid precisely under the assumption that the derivative
cocycle $Df$ associated to $f$ is quasicompact. 
It is believed that those results will have applications in the study of partial and delay
differential equations. In another direction, above described advancements in the multiplicative ergodic theory have also inspired new directions in the study of statistical properties of random
dynamical systems~\cite{DFGTV}.

The present paper is organized as follows. In Section~\ref{pre}, we introduce notation and recall several concepts and useful results that will be used throughout the paper. In Section~\ref{SacSell} we introduce the notion of a Sacker-Sell
spectrum for quasicompact cocycles and describe all possible structures of the spectrum. 
We note that our results and proofs in this section are 
inspired by the results in~\cite{BDV}, where the authors have  developed  the version of the Sacker-Sell theory with respect to the notion of nonuniform hyperbolicity. 
Finally, in Section~\ref{JPSS} we establish the main result of this paper, i.e. the version of the theorem of Johnson, Palmer and Sell for quasicompact cocycles.
Here  our strategy follows the one we  outlined in~\cite{DD} 
(in the simple case of finite-dimensional dynamics) and consists of using deep results of Cao~\cite{C1} on the growth of subadditive functions over semiflows. 

\section{Preliminaries}\label{pre}
\subsection{Uniformly hyperbolic cocycles}
Throughout this paper  $M$ will  be a compact topological space and  $f\colon M \to M$ will be a homeomorphism. Furthermore, assume that  $X$ is a  Banach space and denote by $B(X)$ the space of all bounded operators
 on $X$. Finally, let $\mathbb N_0=\{0, 1, 2, \ldots \}$.   A  map $\cA \colon M\times \N_0 \to B(X)$ is said to be a  \emph{cocycle} over $f$ if:
\begin{enumerate}
\item $\cA(q,0)=\Id$ for each $q\in M$;
\item $\cA(q,n+m)=\cA(f^n(q),m)\cA(q,n)$ for every $q\in M$ and $n,m \in \N_0$;
\item the map $A\colon M \to B(X)$ defined  by
\begin{equation}\label{gen}
 A(q)=\cA(q, 1), \quad  q\in M
\end{equation}
is continuous. 
\end{enumerate}
We recall that the map  $A$ given by~\eqref{gen} is called the   \emph{generator} of the  cocycle $\cA$.
 We say that a cocycle  $\cA$ over $f$ is \emph{uniformly hyperbolic} if:
\begin{enumerate}
\item there exists a family of projections $P(q)$, $q\in M$ satisfying
\begin{equation}\label{proj}
A(q)P(q)=P(f(q))A(q), \quad q\in M
\end{equation}
such that each map $A(q)\rvert \Ker P(q) \colon \Ker P(q)\to \Ker P(f(q))$ is invertible;
\item  $q\mapsto P(q)$ is a continuous map from $M$ to $B(X)$;
\item there exist $D, \lambda >0$ such that for each $q\in M$ and $n\ge 0$
\begin{equation}\label{UH1}
\lVert \cA(q,n)P(q)\rVert \le De^{-\lambda n}
\end{equation}
and
\begin{equation}\label{UH2}
\lVert \cA(q,-n)(\Id-P(q))\rVert \le De^{-\lambda n},
\end{equation}
where $\cA(q,-n)=(\cA(f^{-n}(q),n)\rvert \Ker P(f^{-n}(q)))^{-1}$.
\end{enumerate}
Let $B_X$ denote the closed unit ball in $X$ centered at $0$.
For an operator $A\in B(X)$, we define $\lVert A\rVert_{ic(X)}$ to be an infimum over all $r>0$ with the property that $A(B_X)$ can be covered by finitely many balls of radius $r$.
It is easy to verify that $\lVert A\rVert_{ic(X)} \le \lVert A\rVert$. The following result will play an important role in our arguments. 
\begin{lemma}\label{l1}
 Assume that a cocycle $\cA$ is uniformly hyperbolic and that
 \begin{equation}\label{1}
  \limsup_{n\to \infty} \frac 1n \log \lVert \cA(q, n)\rVert_{ic(X)} \le 0,
 \end{equation}
for some $q\in M$. Then, $\dim \Ker P(q) <\infty$.
\end{lemma}

\begin{proof}
 Suppose  that $\dim \Ker P(q)=\infty$. By Riesz lemma, we can find a sequence $(e_n)_{n\in \N} \subset B_X\cap \Ker P(q)$ such that $\lVert e_n-e_m\rVert \ge 1/2$ for $n\neq m$.
 It follows from~\eqref{UH2} that
 \[
 \lVert \cA(q,n)(e_k-e_l)\rVert \ge \frac{1}{D}e^{\lambda n} \lVert e_k-e_l\rVert \ge \frac{1}{2D}e^{\lambda n},
 \]
for $n\in \N$ and $k, l\in \N$ such that $k\neq l$. Hence, $\cA(q,n)(B_X)$ cannot be covered by  finitely many balls of radius $\frac{1}{4D}e^{\lambda n}$.
Thus, \[ \lVert \cA(q,n)\rVert_{ic(X)} \ge \frac{1}{4D}e^{\lambda n}\] which implies that
\[
  \limsup_{n\to \infty} \frac 1n \log \lVert \cA(q, n)\rVert_{ic(X)} \ge \lambda >0.
\]
This yields  a contradiction with~\eqref{1}.
\end{proof}

\subsection{Quasicompactness}
We first note that since $A$ is continuous and $M$ is compact, we have that 
\[
 \sup_{q\in M} \lVert A(q)\rVert <\infty.
\]
Hence, 
it follows from the subadditive ergodic theorem that for each ergodic $f$-invariant measure $\mu$, there exist 
$\kappa(\mu), \lambda(\mu)\in [-\infty, \infty)$ such that 
\[
\lim_{n\to \infty} \frac 1 n \log \lVert \cA(q, n)x\rVert =\lambda(\mu) \quad \text{and} \quad \lim_{n\to \infty} \frac 1 n \log \lVert \cA(q, n)x\rVert_{ic(X)}=\kappa(\mu),
\]
for $\mu$-a.e. $q\in M$.  We say that $\cA$ is \emph{quasicompact} with respect to $\mu$ if
\begin{equation}\label{quas}
\kappa(\mu)< \lambda (\mu).
\end{equation}
 To the best of our knowledge, the notion of quasicompactness in the context of cocycles over some  measure-preserving dynamical system originated in the work of Thieullen~\cite{PT}.

The following result which is  proved in~\cite[Lemma C.5.]{GTQ} (building on the work Buzzi~\cite{Buzzi}) provides sufficient conditions under which the cocycle is quasicompact. 
\begin{proposition}\label{prop1}
Let $(X', \lvert \cdot \rvert)$ be a Banach space such that $X\subset X'$ and that the inclusion $(X, \lVert \cdot \rVert) \hookrightarrow (X', \lvert \cdot \rvert)$ is compact. Furthermore,
suppose that $A(q)$ can be extended to a bounded operator on $(X', \lvert \cdot \rvert)$ for each $q\in M$. Finally, let $\mu$ be an ergodic, $f$-invariant Borel probability measure 
such that there exist Borel-measurable functions $\alpha, \beta, \gamma \colon M \to (0, \infty)$ with $\gamma$ being $\log$-integrable such that for $\mu$-a.e. $q\in M$ and $x\in X$,
\begin{equation}\label{LY1}
\lVert A(q)x\rVert  \le \alpha(q)\lVert x\rVert +\beta(q) \lvert x\rvert  
\end{equation}
and
\begin{equation}\label{LY2}
\lVert A(q)\rVert  \le \gamma(q).
\end{equation}
Finally, let us assume that 
\begin{equation}\label{gap}
 \int \log \alpha(q) \, d\mu(q) <\lambda(\mu).
\end{equation}
Then, 
\[
 \kappa (\mu) \le \int_M \log \alpha (q)\, d\mu(q).
\]
In particular,  $\cA$ is quasicompact with respect to $\mu$.
\end{proposition}
In the context of   cocycles of the  so-called transfer operators which play a major role in modern dynamical systems theory, 
conditions~\eqref{LY1} and~\eqref{LY2} are called the   strong and weak Lasota-Yorke inequality respectively.  We refer to~\cite{Buzzi, DFGTV, FLQ, GTQ} for details. 

\subsection{Multiplicative ergodic theorem}
Let us state the following version of the multiplicative ergodic theorem.
\begin{theorem}\label{met}
 Assume that $\mu$ is ergodic, $f$-invariant Borel probability measure. Furthermore, suppose that the cocycle $\cA$ is quasicompact with respect to $\mu$. Then, we have the following:
 \begin{enumerate}
  \item there exists $l\in [1, \infty]$ and a sequence of numbers $(\lambda_i)_{i=1}^l$ such that 
  \[
   \lambda(\mu)=\lambda_1 >\lambda_2 >\ldots >\lambda_i > \ldots >\kappa(\mu).
  \]
Furthermore, if $l=\infty$ we have that $\lim_{i\to \infty} \lambda_i =\kappa(\mu)$;
\item for $\mu$-a.e. $q\in M$ there exists a unique and measurable decomposition
\[
 B=\bigoplus_{i=1}^l E_i(q) \oplus F(q),
\]
where $E_i$ are closed, finite-dimensional subspaces of $X$ and $A(q)E_i(q)=E_i(f(q))$. Furthermore, $F(q)$ are closed subspaces of $X$ and $A(q)F(q)\subset F(f(q))$;
\item for each $1\le i \le l$ and $v\in E_i(q)\setminus \{0\}$, we have 
\[
 \lim_{n\to \infty} \frac 1 n \log \lVert \cA(q, n)v\rVert =\lambda_i.
\]
In addition, for every $v\in F(x)$,
\[
 \limsup_{n\to \infty} \frac 1 n \log \lVert \cA(q, n)v\rVert \le \kappa(\mu).
\]

 \end{enumerate}
The numbers $\lambda_i$ are called \emph{Lyapunov exponents} of the cocycle $\cA$ with respect to $\mu$.

\end{theorem}
Theorem~\ref{met} was established by   Froyland, Lloyd and Quas~\cite{FLQ}, building on the earlier work of Thieullen~\cite{PT} who has obtain analogous result but under additional assumption that
the cocycle $\cA$ is injective. We also note that the version of Theorem~\ref{met} established in~\cite{FLQ} works under milder assumptions both for the base space $(M,f)$ and  for the cocycle 
$\cA$. In particular, for the conclusion of Theorem~\ref{met} to hold it is not necessary to assume that $A$ is a continuous map. Subsequent generalizations of the main result in~\cite{FLQ}
were obtained by Gonz\'alez-Tokman and Quas~\cite{GTQ, GTQ1} and Blumenthal~\cite{AB}.

\subsection{Growth of subadditive functions over semiflows}
In this subsection we recall the very useful result of Cao~\cite{C1}. A sequence of continuous functions $(F_n)_{n\in \N}$, $F_n \colon M \to \R$ is said to be  \emph{subadditive} if
\[
 F_{n+m}(q) \le F_m(f^n(q))+F_n(q), \quad \text{for $q\in M$ and $n, m\in \N$.}
\]
Let $\mathcal E(f)$ denote the set of all ergodic $f$-invariant measures. The following result is due to Cao~\cite[Theorem 4.2.]{C1}.

\begin{theorem}\label{t0}
 There exists $\nu \in \mathcal E(f)$ such that
 \[
 \lim_{n\to \infty} \frac 1 n  \max_{q\in M} F_n(q)=\Lambda(\nu)=\max_{\mu \in \mathcal E(f)} \Lambda (\mu),
 \]
where $\Lambda (\mu)=\inf_{n\in \N}\frac 1n \int_M F_n\, d\mu$.
\end{theorem}

\section{Sacker-Sell spectrum for quasicompact cocycles}\label{SacSell}
In this section we introduce the notion of a Sacker-Sell spectrum for quasicompact cocycles and describe all possible structures of this spectrum. As we have already mentioned, our arguments follow closely the approach
developed in~\cite{BDV}.

>From now on, we consider the cocycle $\cA$ over $f$ such that~\eqref{quas} holds for each $\mu \in \mathcal E(f)$. The following simple consequence of Proposition~\ref{prop1} provides sufficient conditions under
which this is satisfied. 
\begin{proposition}
 Let $(X, \lVert \cdot \rVert)$ and $(X', \lvert \cdot \rvert)$ be as in the statement of Proposition~\ref{prop1}. 
 Furthermore, assume that there exist Borel measurable functions $\alpha, \beta, \gamma \colon M \to (0, \infty)$ with $\gamma$ being $\log$-integrable such that:
 \begin{enumerate}
  \item \eqref{LY1} and~\eqref{LY2} hold for each $q\in M$ and $x\in X$;
  \item \eqref{gap} holds for each $\mu \in \mathcal{E}(f)$.
 \end{enumerate}
Then, \eqref{quas} holds for each $\mu \in \mathcal E(f)$.
\end{proposition}

\begin{proof}
 By applying Proposition~\ref{prop1} for an arbitrary $\mu \in \mathcal E(f)$, we obtain that
 \[
  \kappa(\mu)\le \int_M \log \alpha(q)\, d\mu(q) <\lambda(\mu),
 \]
and thus~\eqref{quas} holds. 
\end{proof}

\begin{proposition}
 We have  that $\max_{\mu \in \mathcal E(f)} \kappa (\mu)$ exists. 
\end{proposition}

\begin{proof}
 This is a direct consequence of Theorem~\ref{t0} applied to a subadditive sequence 
 \[
  F_n(q)=\lVert \cA(q, n)\rVert_{ic(X)}, \quad q\in M, \ n\in \N.
 \]

\end{proof}
Let
\[
\kappa :=\max_{\mu \in \mathcal E(f)}\kappa (\mu)\in [-\infty, \infty)
\]
Furthermore, 
for any $a\in \R$ we can define a  cocycle $\cA_a$ over $f$ by
\[
 \cA_a(q, n)=e^{-an}\cA(q,n), \quad q\in M, \ n\ge 0.
\]
Finally, we define the  Sacker-Sell spectrum $\Sigma=\Sigma (\cA)$ of $\cA$ by
\[
 \Sigma=\bigg{\{} a\in \R :  \text{$a> \kappa$   and $\cA_a$ is not uniformly hyperbolic} \bigg{\}}\subset (\kappa, \infty).
\]
The main aim of this section is  to describe all possible structures of $\Sigma$.
We first introduce some useful notation. For $a\in \R$ and $q\in M$, set
\[
S_a(q)=\bigg{\{} x\in X: \sup_{n\ge 0}  e^{-an}\lVert \cA(q, n)x\rVert <\infty \bigg{\}}.
\]
Moreover, let $U_a(q)$ denote the set of all $x\in X$ for which there exists a sequence $(x_n)_{n\le 0} \subset X$ such that
$\sup_{n\le 0} (e^{-an} \lVert x_n\rVert) <\infty$ and
\[
x_n=A(f^{n-1}(q))x_{n-1} \quad  \text{for $n\le 0$.}
\]
It is easy to verify  that if $\cA_a$ is uniformly hyperbolic with  projections $P(q)$, $q\in M$ that then
\[
\Ima P(q)=S_a(q) \quad \text{and} \quad \Ker P(q)=U(q),
\]
for every $q\in M$. Moreover, for $a_1<a_2$ we have
\[
S_{a_1}(q) \subset S_{a_2}(q) \quad \text{and} \quad U_{a_2}(q) \subset U_{a_1}(q),
\]
for each $q\in M$. We now collect basic properties of the Sacker-Sell spectra which can be found in~\cite{SS}. Although the original work of Sacker and Sell considers only the case of finite-dimensional
 and invertible dynamics, all the auxiliary results we list below
can be easily  proved by repeating the arguments in~\cite{SS} (see also~\cite{BDV}).

\begin{lemma}\label{l2}
$\Sigma$ is a closed set. More precisely, for $a> \kappa$ such that $a\notin \Sigma$, there exists $\epsilon >0$ such that for each $b\in  (a-\epsilon, a+\epsilon)$ we have that:
\begin{enumerate}
 \item $b\notin \Sigma$;
 \item for every $q\in M$, 
 \[
S_b(q)=S_a(q) \quad \text{and} \quad  U_b(q)=U_a(q).
\]
\end{enumerate}
\end{lemma}
Note that for $a>\kappa$ such that $a\notin \Sigma$,  $\dim U_a(q)$ doesn't depend on $q$ and thus  we can simply write $\dim U_a$. Moreover, 
it follows from Lemma~\ref{l1} that  $\dim U_a <\infty$ for each such $a$.
\begin{lemma}\label{sf}
 Take $a_2>a_1>\kappa$ such that $a_1, a_2\notin \Sigma$. Then, the following properties are equivalent:
 \begin{enumerate}
  \item $[a_1, a_2]\subset (\kappa, \infty)\setminus \Sigma$;
  \item $\dim U_{a_1}=\dim U_{a_2}$.
 \end{enumerate}

\end{lemma}

\begin{lemma}\label{sf1}
 We have that $\sup \Sigma <\infty$. 
\end{lemma}
\begin{proof}
 Since $A$ is continuous and $M$ is compact, there exists $C>0$ such that $\lVert A(q)\rVert \le C$ for $q\in M$. Hence,
 \begin{equation}\label{ub}
  \lVert \cA(q, n)\rVert \le C^n \quad \text{for $q\in M$ and $n\in \N$.}
 \end{equation}
It follows from~\eqref{ub} that for each $a>\log C$, the cocycle $\cA_a$ is uniformly hyperbolic with projections $P(q)=\Id$, $q\in M$. This readily implies the desired conclusion.
 
\end{proof}

The following lemma is crucial for our purposes.

\begin{lemma}\label{FN}
Let $c>\kappa$ such that $c\notin \Sigma$. Then, $\Sigma \cap [c, \infty)$ is the union of finitely many closed intervals.
\end{lemma}

\begin{proof}
Let $d=\dim U_c$ and  assume that $\Sigma \cap [c,+\infty)$ contains at least $d+2$ disjoint closed intervals $I_i=[\alpha_i, \beta_i]$, for $i=1,\ldots, d+2$, where
\[
\alpha_1 \le \beta_1 <\alpha_2 \le \beta_2 <\cdots <\alpha_{d+2} \le \beta_{d+2}< +\infty.
\]
For $i\in \{1, \ldots , d+1\}$, take $c_i \in (\beta_i, \alpha_{i+1})$. By Lemma~\ref{sf},
\[
d>\dim U_{c_1} >\dim U_{c_2} >\cdots >\dim U_{c_{d+1}},
\]
which is clearly impossible.
\end{proof}
The following is the main result of this section. 
\begin{theorem}\label{SSS}
One of the following alternatives holds:
\begin{enumerate}
\item
$\Sigma=\emptyset$;
\item
$\Sigma= \bigcup_{n=1}^k [a_n,b_n]$,  for some numbers
\begin{equation}\label{*N}
b_1 \ge a_1>b_2 \ge a_2 > \cdots > b_k \ge a_k > \kappa;
\end{equation}
\item
$\Sigma= \bigcup_{n=1}^{k-1} [a_n,b_n] \cup (\kappa,b_k]$, for some numbers $a_n$ and $b_n$ as in~\eqref{*N};
\item
$\Sigma= \bigcup_{n=1}^\infty [a_n, b_n]$,  for some numbers
\begin{equation}\label{WW}
b_1 \ge a_1>b_2 \ge a_2 >\cdots
\end{equation}
with $\lim_{n\to +\infty} a_n=\kappa$;
\item
$\Sigma=\bigcup_{n=1}^\infty [a_n, b_n] \cup (\kappa,b_\infty]$,  for some numbers $a_n$ and $b_n$ as in~\eqref{WW} with $b_\infty:=\lim_{n\to +\infty} a_n >\kappa$.
\end{enumerate}
\end{theorem}

\begin{proof}
Since $\Sigma$ is closed and $\Sigma \neq (\kappa, \infty)$ (see Lemma~\ref{sf1}), if $\Sigma$ is nonempty and has finitely many connected components, then 
it has one of the forms in alternatives $2$ and~$3$. 

Now we consider the case when $\Sigma$ has infinitely many connected components. Namely, assume that $\Sigma$ is not given by one of the first three alternatives in the theorem and take $c_1\notin \Sigma$. By Lemma~\ref{FN}, the set $\Sigma \cap (c_1,+\infty)$ consists of finitely many disjoint closed intervals $I_1, \ldots , I_k$.
We note that $\Sigma \cap (\kappa, c_1) \ne \emptyset$, since otherwise we would have $\Sigma=I_1 \cup \cdots \cup I_k$, which contradicts to our assumption.
Now we observe that there exists $c_2 <c_1$ such that $c_2\notin \Sigma$ and $(c_2, c_1)\cap \Sigma \neq \emptyset$. Indeed, otherwise we would
have $(\kappa, c_1)\cap \Sigma= (\kappa, a]$ for some $a<c_1$ and thus,
\[
\Sigma=(\kappa, a] \cup I_1 \cdots \cup I_k,
\]
which again contradicts to our assumption. Proceeding inductively, we obtain a decreasing sequence $(c_n)_{n \in \N}$ such that $c_n \notin\Sigma$ and $(c_{n+1}, c_n) \cap \Sigma \neq \emptyset$ for each $n\in \N$. Now there are two possibilities:
either $\lim_{n\to+\infty} c_n=\kappa$ or $\lim_{n\to +\infty} c_n=b_\infty$ for some $b_\infty >\kappa$. In the first case, it follows from Lemma~\ref{FN} that $\Sigma$ is given by alternative~5 in the theorem. In the second case, it follows from Lemma~\ref{FN} that
\[
(b_\infty,+\infty)\cap \Sigma=I_1\cup \bigcup_{n=2}^\infty [a_n, b_n],
\]
where $I_1=[a_1, b_1]$ or $I_1=[a_1,+\infty)$, for some sequences $(a_n)_{n \in \N}$ and $(b_n)_{n \in \N}$ as in~\eqref{WW} with $b_\infty=\lim_{n\to +\infty} a_n$. Again by Lemma~\ref{FN}, we have $(\kappa, b_\infty] \subset \Sigma$ and so $\Sigma$ is given by the last alternative in the theorem.
\end{proof}
We note that the structure of the Sacker-Sell spectrum for \emph{compact} operators on an arbitrary Banach space under the additional assumption that $f$ has a fixed or a periodic orbit was discussed in~\cite{LM}. There, a result similar to Theorem~\ref{SSS} was established. In particular, it was proved that
the last alternative in the statement of Theorem~\ref{SSS} never occurs in that setting. 

\section{A version of the theorem of Johnson, Palmer and Sell}\label{JPSS}
We begin this section by noting that our assumption that~\eqref{quas} holds for any $\mu \in \mathcal E(f)$, enable us to apply Theorem~\ref{met}  for each $\mu \in \mathcal E(f)$. Hence, we obtain the set of Lyapunov exponents of our cocycle $\cA$ with respect an arbitrary measure in
$\mathcal E(f)$. 
The following  version of the theorem of Johnson, Palmer and Sell~\cite{JPS} for quasicompact cocycles is the main result of our paper.

\begin{theorem}\label{jps}
Assume that $\Sigma \neq \emptyset$. 
 The endpoints $a_n$ and $b_n$ for $n\in \N$ of spectral intervals contained in $\Sigma$ are Lyapunov exponents of the cocycle $\cA$ with respect to some ergodic $f$-invariant measure.
\end{theorem}

\begin{proof}
 Take a spectral  interval $[a_m, b_m]\subset \Sigma$ and let us prove that  $b_m$ is a Lyapunov exponent of $\cA$
 with respect to some ergodic $f$-invariant measure. The argument for $a_m$ is completely analogous. Assume that $b_m$ is not a Lyapunov  exponent of $\cA$  and 
 take  $c=b_m+\epsilon$, where  $\epsilon >0$. For a sufficiently small $\epsilon >0$, 
  $c\notin \Sigma$ and therefore   the cocycle $\cA_{c}$ is uniformly hyperbolic with respect to projections $P(q)$, $q\in M$ that satisfy
 \[
  \Ima P(q)=S_{c}(q) \quad \text{and} \quad \Ker P(q)=U_c(q).
 \]
 In addition, by Lemma~\ref{sf} the subspaces $S_c(q)$ and $U_c(q)$ don't depend on the choice of $c$.
 In particular, there exist $D, \lambda >0$ such that
\[
 \lVert e^{-cn}\cA(q,n)P(q)\rVert \le De^{-\lambda n}, \quad \text{for $n\ge 0$},
\]
which implies that for $v\in S_{c}(q)$,
\[
 \limsup_{n\to \infty} \frac 1 n\log \lVert \cA(q,n)v\rVert \le c-\lambda <c.
\]
Letting $\epsilon \to 0$, we have that  $c\to b_m$ and therefore
\begin{equation}\label{in}
  \limsup_{n\to \infty} \frac 1 n\log \lVert \cA(q,n)v\rVert \le b_m.
\end{equation}
Set 
\[
 F_n(q)=\log \lVert \cA(q, n)P(q)\rVert \quad \text{for $q\in M$ and $n\in \N$.}
\]
It follows from the continuity of $A$ and the map $q\mapsto P(q)$ that $F_n\colon M \to \mathbb R$ is continuous map for each $n\in \N$.
\begin{lemma}
 The sequence $(F_n)_{n\in \mathbb N}$ is subadditive.
\end{lemma}

\begin{proof}[Proof of the lemma]
It follows from~\eqref{proj} that 
\[
 \begin{split}
  \lVert \cA(q, n+m)P(q)\rVert &=\lVert \cA(q, n+m)P(q)^2\rVert \\
  &=\lVert \cA(f^n(q), m)\cA(q, n)P(q)P(q)\rVert \\
  &=\lVert \cA(f^n(q), m)P(f^n(q))\cA(q, n)P(q)\rVert \\
  &\le \lVert \cA(f^n(q), m)P(f^n(q))\rVert \cdot \lVert \cA(q, n)P(q)\rVert,
 \end{split}
\]
for $q\in M$ and $n, m\in \N$ which immediately implies the desired conclusion.

\end{proof}
For $\mu \in \mathcal E(f)$, set 
\[
 \Lambda(\mu)=\inf_{n\in \mathbb N} \frac 1 n \int_M F_n\, d\mu.
\]
It follows from Kingman's subadditive ergodic theorem that
\[
 \Lambda(\mu) =\lim_{n\to \infty}\frac{1}{n}F_n(q) \quad \text{for $\mu$-a.e. $q\in M$.}
\]

\begin{lemma}\label{le}
 $\Lambda(\mu)$ is a Lyapunov exponent of $\cA$ with respect to $\mu$.
\end{lemma}

\begin{proof}[Proof of the lemma]
Using the arguments similar to those we used to establish~\eqref{in}, one can easily prove that for each nonzero $v\in S_c(q)\cap U_{a_m-\epsilon}(q)$ 
(we note that such $v$ exists by Lemma~\ref{l2}), we have
\[
 \limsup_{n\to \infty} \frac 1 n\log \lVert \cA(q,n)v\rVert \ge a_m, 
\]
which immediately implies that $\Lambda(\mu)\ge a_m$. Hence, \[\Lambda(\mu) \ge a_m >\kappa \ge \kappa(\mu).\]
Let $\lambda_1>\lambda_2 >...$ denote (distinct) Lyapunov exponent of $\cA$ with respect to $\mu$ and assume that $\Lambda(\mu)\neq \lambda_j$ for each $j$. Choose $i$ such that
$\Lambda(\mu) \in (\lambda_{i+1}, \lambda_i)$. Then, for $\mu$-a.e. $q\in M$ and  each $v\in \Ima P(q)$, we have 
\[
 \lim_{n\to \infty} \frac 1 n \log \lVert \cA(q, n)v\rVert \le \lim_{n\to \infty} \frac 1 n \log \lVert \cA(q, n)P(q)\rVert<\lambda_i,
\]
and thus it follows from Theorem~\ref{met} that
\[
  \lim_{n\to \infty} \frac 1 n \log \lVert \cA(q, n)v\rVert \le \lambda_{i+1}
\]
By~\cite[Proposition 14.]{FLQ}, we have  that 
\[
 \lim_{n\to \infty}\frac 1 n \log \lVert \cA(q, n)\rvert \Ima P(q)\rVert \le \lambda_{i+1},
\]
which together with the continuity of the map $q\mapsto P(q)$  implies that $\Lambda(\mu)\le \lambda_{i+1}$. Hence, we have obtained an contradiction.

\end{proof}
We now note that~\eqref{in}, Lemma~\ref{le} and our assumption on $b_m$ imply that $\Lambda(\mu)<b_m$ for each $\mu \in \mathcal E(f)$. By Theorem~\ref{t0}, we have that
\[
 \lim_{n\to \infty} \frac 1 n\max_{q\in M}F_n(q) <b_m,
\]
which readily implies that there exist $D, \lambda>0$ such that
\begin{equation}\label{sa1}
 \lVert e^{-b_mn}\cA(q, n)P(q)\rVert \le De^{-\lambda n} \quad \text{for $q\in M$ and $n\ge 0$.}
\end{equation}
Similarly, one can show that there exist $D', \lambda'>0$ such that
\begin{equation}\label{sa2}
 \lVert e^{b_mn}\cA(q, -n)P(q)\rVert \le D'e^{-\lambda' n} \quad \text{for $q\in M$ and $n\ge 0$.}
\end{equation}
It follows from~\eqref{sa1} and~\eqref{sa2} that $\cA_{b_m}$ is uniformly hyperbolic. Therefore, $b_m \notin \Sigma$ which yields a contradiction.

\end{proof}
We emphasize that Theorem~\ref{jps} says nothing about $b_\infty$. In the case when $f$ is uniquely ergodic (which means that $\mathcal E(f)$ consists of only one measure), $b_\infty$
cannot be realized as  a Lyapunov exponent since   Lyapunov exponents with respect to a given measure $\mu$  can only accumulate in $\kappa (\mu)$ (see Theorem~\ref{met}).
As for what happens in general case, it is unclear to me. 

\section{acknowledgements}
I would like to express my sincere gratitude to the anonymous referee of my paper~\cite{DD} who in his/hers report wrote the following: ``It should be  possible to prove this result
for infinite dimension cocycle for Sacker-Sell spectrum  using the results in paper~\cite{C1}". It is fair to say that this paper wouldn't have existed without this gesture of extreme generosity.
I would also like to thank Luis Barreira and Claudia Valls for useful comments on the first version of this paper.

\bibliographystyle{amsplain}

\end{document}